\documentclass[12pt,fleqn]{conm-p-l}

\usepackage{amsmath,amsthm,amssymb,comment,fullpage}

\usepackage{mathabx}
\usepackage{times}
\usepackage[T1]{fontenc}
\usepackage{mathrsfs}
\usepackage{latexsym}
\usepackage{epsfig}
\usepackage{graphicx}
\usepackage{caption}
\usepackage{subcaption}

\usepackage{hyperref, amsmath, amsthm, amsfonts, amscd, flafter,epsf}

\usepackage{amsmath,amsfonts,amsthm,amssymb,amscd}
\input amssym.def
\input amssym.tex
\usepackage{color}
\usepackage{hyperref}
\usepackage{url}

\newcommand{\burl}[1]{\textcolor{blue}{\url{#1}}}

\numberwithin{equation}{section}

\newtheorem{thm}{Theorem}[section]

\newtheorem{cor}[thm]{Corollary}
\newtheorem{lem}[thm]{Lemma}
\newtheorem{prop}[thm]{Proposition}

\theoremstyle{plain}

\newtheorem{lemma}[thm]{Lemma}

\newtheorem{remark}[thm]{Remark}

\newcommand\be{\begin{equation}}
\newcommand\ee{\end{equation}}
\newcommand\bea{\begin{eqnarray}}
\newcommand\eea{\end{eqnarray}}
\newcommand\bi{\begin{itemize}}
\newcommand\ei{\end{itemize}}
\newcommand\ben{\begin{enumerate}}
\newcommand\een{\end{enumerate}}
\newcommand\bc{\begin{center}}
\newcommand\ec{\end{center}}
\newcommand\ba{\begin{array}}
\newcommand\ea{\end{array}}

\newcommand{\R}{\ensuremath{\mathbb{R}}}

\newcommand{\twocase}[5]{#1 \begin{cases} #2 & \text{{\rm #3}}\\ #4
&\text{{\rm #5}} \end{cases}   }
\newcommand{\threecase}[7]{#1 \begin{cases} #2 & \text{{\rm #3}}\\ #4
&\text{{\rm #5}}\\ #6 & \texttt{{\rm #7}} \end{cases}   }

\newcommand{\kkot}[1]{ \frac{\sin \pi {#1} }{\pi {#1} } }
\newcommand{\hphi}{\widehat{\phi}}  

\newcommand{\innerproduct}[1]{\ensuremath{\langle #1 \rangle}}
\newcommand{\norm}[1]{\left\lVert#1\right\rVert}
\newlength{\arrow}
\title[Optimal Test Functions for Bounding Average Rank]{Determining Optimal Test Functions for Bounding the Average Rank in Families of $L$-Functions}

\author{Jesse Freeman}
\address{Department of Mathematics and Statistics, Williams College, Williamstown, MA 01267}
\email{\textcolor{blue}{\href{mailto:jbf1@williams.edu}{jbf1@williams.edu}}}

\author{Steven J. Miller}
\email{\textcolor{blue}{\href{mailto:sjm1@williams.edu}{sjm1@williams.edu}},  \textcolor{blue}{\href{Steven.Miller.MC.96@aya.yale.edu}{Steven.Miller.MC.96@aya.yale.edu}}}
\address{Department of Mathematics and Statistics, Williams College, Williamstown, MA 01267}

\thanks{This research was supported by NSF grants DMS1265673 and Williams College. We thank Eyvi Palsson, Mihai Stoiciu, and the referee for helpful comments and discussions.}



\date{\today}

\begin{document}

\maketitle

\begin{abstract} Given an $L$-function, one of the most important questions concerns its vanishing at the central point; for example, the Birch and Swinnerton-Dyer conjecture states that the order of vanishing there of an elliptic curve $L$-function equals the rank of the Mordell-Weil group. The Katz and Sarnak Density Conjecture states that this and other behavior is well-modeled by random matrix ensembles. This correspondence is known for many families when the test functions are suitably restricted. For appropriate choices, we obtain bounds on the average order of vanishing at the central point in families. In this note we report on progress in determining the optimal test functions for the various classical compact groups for different support restrictions, and discuss how this relates to improved rank bounds.
\end{abstract}



\section{Introduction}

\subsection{Background}

While the importance of random matrices in mathematics and related disciplines had been noticed at least as early as Wishart's work \cite{Wis} in the late 1920s, for us in number theory the story begins with the connections observed by Montgomery and Dyson \cite{Mon} in the 1970s. Montgomery was studying the pair-correlation of zeros of the Riemann zeta function, and the behavior was identical to that of certain random matrix ensembles which had been extensively studied due to their applicability in nuclear physics. Briefly, characteristic polynomials (and their eigenvalues) of the classical compact groups have been observed to model well $L$-functions (and their critical zeros). While we will concentrate on low-lying zeros, i.e., zeros near the central point, in families of $L$-functions, there is an extensive literature on other statistics, including $n$-level correlations \cite{Hej,Mon,RS}, spacings \cite{Od1,Od2}, and moments \cite{CFKRS}. See \cite{FM,Ha} for a brief history of the subject and \cite{Con,For,KaSa1,KaSa2,KeSn1,KeSn2,KeSn3,Meh,MT-B,T}) for some articles and textbooks on the connections.

In many of the earlier works on the correspondences between the two subjects, the  statistics studied were insensitive to the behavior of finitely many zeros. This led to the introduction of a new statistic, the $n$-level density, as often the zeros near the central point are related to important arithmetic quantities, with the Birch and Swinnerton-Dyer conjecture (stating that the order of vanishing of the $L$-function at the central point equals the rank of the Mordell-Weil group of rational solutions) the most famous example. In this paper we concentrate on the $1$-level density, which we define in detail in \S\ref{sec:defn1leveldensity}. We report on recent results from the first author's honors thesis at Williams College, supervised by the second author, where building on methods introduced in \cite{ILS} optimal test functions are constructed for various statistics for different support ranges. The main application of these theorems are improved estimates on the average vanishing at the central point for families of $L$-functions. In addition to being of general interest, such results have important applications (for example, in \cite{IS} good estimates here are connected to the Landau-Siegel zero question).

In the arguments below we concentrate on the limiting behavior. An important topic for future research is to include lower order terms and determine the optimal test functions for various regimes where the limiting behavior has not yet been reached. These regimes are quite important as they are the ones that can be investigated numerically, and often the data gathered is at odds with the limiting predictions as the rate of convergence is abysmally slow. The prime example is that of whether or not their is excess rank in families of elliptic curves (see \cite{BMSW} for a nice summary of data and conjectures); while earlier investigations indicated that such bias might persist, later studies \cite{W} went far enough to see the average rank drop, and new random matrix models have been introduced that have the correct limiting behavior and successfully model the observed behavior for small conductors \cite{DHKMS1,DHKMS2}. There are now many results on lower order terms in families, such as \cite{HKS,MMRW,Mil2,Yo1}, and the hope is that the methods of this paper can be extended to include these to refine estimates for finite conductors. \\ \

\emph{On a personal note, the second author investigated questions on rates of convergence with Ram Murty in \cite{MM} (explicitly, proving effective bounds on families of elliptic curves modulo $p$ (for $p$ prime tending to infinity) obeying the Sato-Tate Law). It is a pleasure to dedicate this work to him on the occasion of his 60\textsuperscript{th}  birthday, and we hope to report on extending our result to lower order terms before his next big celebration!}


\subsection{$n$-Level Density}\label{sec:defn1leveldensity}

As alluded to above, the behavior of zeros far from the central point exhibit a remarkable universality across $L$-functions. Unfortunately it is significantly harder to study one $L$-function's zeros near the central point. The reason is that there are only a few normalized zeros near the central point, and there is thus no possibility of averaging if we restrict ourselves to just one object (in the extreme case of whether or not the $L$-function vanishes, we just have a `yes-no' question). To  make progress, we instead study a family of $L$-functions. The Katz-Sarnak philosophy \cite{KaSa1,KaSa2} states that the behavior of a family of $L$-functions should be well-modeled by a corresponding classical compact groups, with the conductor in the family tending to infinity playing the same roll as the growing matrix size; for alternative approaches to modeling the behavior of zeros, see \cite{CFZ1,CFZ2,GHK}.

We briefly describe the main statistic studied, the $n$-level density (though we will report on progress on the 1-level only here, see \cite{F} for additional results). For ease of exposition we assume the Generalized Riemann Hypothesis, so given an $L(s,f)$ all the zeros are of the form $1/2 + i\gamma_{j;f}$ with $\gamma_{j;f}$ real. Our statistic makes sense more generally, but we lose the interpretation of ordered zeros and connections with nuclear physics; the main use of GRH is to extend the support calculation for many of the number theory computations. We assume the reader is familiar with $n$-level densities; for more detail on these statistics see the seminal work by Iwaniec, Luo and Sarnak \cite{ILS}, who introduced them, or \cite{AAILMZ} for an expanded discussion (which formed the basis of the quick summary below).

Let $\phi(x) = \prod_{j=1}^n \phi_j(x_j)$ where each $\phi_j$ is an even Schwartz function such that the Fourier transforms \be \widehat{\phi_j}(y) \ := \ \int_{-\infty}^\infty \phi_j(x) e^{-2\pi i xy} dx\ee are compactly supported. The $n$-level density for $f$ with test function $\phi$ is
\begin{eqnarray}
D_n(f,\phi) & \ = \ &  \sum_{j_1, \dots, j_n \atop j_\ell \neq j_m} \phi_1\left(L_f \gamma_{j_1;f}\right) \cdots \phi_n\left(L_f
\gamma_{j_n;f}\right),
\end{eqnarray} where $L_f$ is a scaling parameter which is frequently related to the conductor.
Given a family $\mathcal{F} = \cup_N \mathcal{F}_N$ of $L$-functions with conductors tending to infinity, the $n$-level density $D_n(\mathcal{F},\phi,w)$ with test function $\phi$ and non-negative weight function $w$ is defined by \be D_n(\mathcal{F},\phi,w) \ := \ \lim_{N \to \infty} \frac{\sum_{f\in \mathcal{F}_N} w(f) D_n(f,\phi)  }{ \sum_{f\in \mathcal{F}_N} w(f)}. \ee Frequently one chooses $\mathcal{F}_N$ to be either all forms with conductor equal to $N$, or conductor at most $N$.

Unlike the $n$-level correlations of a family, which have a universal limit as the height of the zero tends to infinity, Katz and Sarnak \cite{KaSa1, KaSa2} proved that the $n$-level density is different for each classical compact group. They were able to obtain closed form determinant expansions; while these expressions can be hard to use for $n \ge 2$ (see \cite{HM} for a discussion on the benefits of an alternative), they are very easy to use for the 1-level.

Let $K(y) := \kkot{y}$ and $K_\epsilon(x,y) := K(x-y) + \epsilon K(x+y)$ for $\epsilon = 0, \pm 1$. If $G_N$ is the family of $N\times N$ unitary, symplectic or orthogonal families (split or not split by sign), the $n$-level density for the eigenvalues converges as $N\to\infty$ to \bea & & \int_{-\infty}^\infty \cdots \int_{-\infty}^\infty \phi(x_1,\dots,x_n) W_{n,G}(x_1,\dots,x_n) dx_1 \cdots dx_n \nonumber\\ & & \ \ \ \ \  \ = \ \int_{-\infty}^\infty \cdots \int_{-\infty}^\infty \hphi(y_1,\dots,y_n) \widehat{W}_{n,G}(y_1,\dots,y_n) dy_1 \cdots dy_n, \eea where
\begin{eqnarray}\label{eqdensitykernels}
W_{m,{\rm SO(even)}}(x) &\ =\ & \det (K_1(x_i,x_j))_{i,j\leq m} \nonumber\\
W_{m,{\rm SO(odd)}}(x) & = & \det (K_{-1}(x_i,x_j))_{i,j\leq
m}  + \sum_{k=1}^m
\delta(x_k) \det(K_{-1}(x_i,x_j))_{i,j\neq k} \nonumber\\
W_{m, {\rm O}}(x) & = & \frac12  W_{m,{\rm SO(even)}}(x) + \frac12  W_{m,{\rm SO(odd)}}(x)
\nonumber\\ W_{m, {\rm U}}(x) & = & \det (K_0(x_i,x_j))_{i,j\leq
m} \nonumber\\ W_{m, {\rm Sp}}(x) &=& \det(K_{-1}(x_i,x_j))_{i,j\leq m}.
\end{eqnarray}

While these densities are all different, for the 1-level density with test functions whose Fourier transforms are supported in $(-1, 1)$, the three orthogonal flavors cannot be distinguished from each other in this regime, though they can be distinguished from the unitary and symplectic.

In many of the calculations it is convenient to shift to the Fourier transform side. Letting \be \threecase{\eta(u) \ = \ }{1}{if $|u| < 1$}{1/2}{if $|u| = 1$}{0}{if $|u| > 1$} \ee and $\delta_0$ denote the standard Dirac Delta functional, for the $1$-level densities we have
\begin{eqnarray}
\widehat{W}_{1,{\rm SO(even)}}(u) & = & \delta_0(u) + \frac12  \eta(u)
\nonumber\\ \widehat{W}_{1,O}(u) & = & \delta_0(u) + \frac12
\nonumber\\ \widehat{W}_{1,{\rm SO(odd)}}(u) & = & \delta_0(u) - \frac12
\eta(u) + 1 \nonumber\\ \widehat{W}_{1,Sp}(u) & = & \delta_0(u) -
\frac12  \eta(u) \nonumber\\ \widehat{W}_{1,U}(u) & = & \delta_0(u),
\end{eqnarray} Note that the first three densities agree for $|u| < 1$ and split (i.e., become distinguishable) for $|u| \geq 1$; alternatively, one could use the 2-level density which suffices to distinguish all candidates for arbitrarily small support (see \cite{Mil2}).

As stated earlier, the Katz-Sarnak Density Conjecture is that the behavior of zeros near the central point in a family of $L$-functions (as the conductors tend to infinity) agrees with the behavior of eigenvalues near 1 of a classical compact group (as the matrix size tends to infinity). There is now  an extensive body of work supporting this for numerous families and various levels of support, including Dirichlet characters, elliptic curves, cuspidal newforms, symmetric powers of ${\rm GL}(2)$ $L$-functions, and certain families of ${\rm GL}(4)$ and ${\rm GL}(6)$ $L$-functions; see for example \cite{DM1, DM2, ER-GR, FiM, FI, Gao, Gu, HM, HR, ILS, KaSa2, LM, Mil1, MilPe, OS1, OS2, RR, Ro, Rub, Ya, Yo2}. This correspondence between zeros and eigenvalues allows us, at least conjecturally, to assign a definite symmetry type to each family of $L$-functions (see \cite{DM2, ShTe} for more on identifying the symmetry type of a family).


\subsection{Main Result}\label{sec:mainresult}

One of the most important applications of the $n$-level density is to estimate the average order of vanishing of $L$-functions $L(s,f)$ at the central point in a family; this is the analytic rank, and is denoted ${\rm Rank}(f)$. While in some families it is natural to use slowly varying weights (such as the Petersson weights for families of holomorphic cusp forms), with additional work these weights can often be removed (see \cite{ILS}).

If we assume GRH for our family of $L$-functions, then all critical zeros have real part 1/2. Further, if our test function $\phi$ is non-negative, then in the $1$-level density we obtain an upper bound for the average rank by removing the contribution from all zeros not at the central point: \be\label{eq:formulatoboundrank} \lim_{N\to\infty} \frac{\sum_{f\in \mathcal{F}_N} w(f) {\rm Rank}(f) \phi(0)}{ \sum_{f\in \mathcal{F}_N} w(f)} \ \le \ \int_{-\infty}^\infty \phi(x) W_{1,G}(x) dx \ = \ \int_{-\infty}^\infty \hphi(y) \widehat{W}_{1,G}(y) dy. \ee

In practice, we can only establish the $n$-level density for test functions with restricted support. On the number theory side, the goal is to verify the correspondence for as large of support as possible, as we can then use \eqref{eq:formulatoboundrank} to bound the rank: \be\label{rankboundplusphicriterion}\lim_{N\to\infty} {\rm AveRank}(\mathcal{F}_N) \ \le \ \frac{\int_{-\infty}^\infty \hphi(y) \widehat{W}_{1,G}(y) dy}{\phi(0)}. \ee Note that instead of trying to increase the support for the 1-level density we could shift to studying higher level densities. While this gives us better bounds for high vanishing at the central point, the probability of vanishing to order $r$ or higher decays like $c_n / r^n$, unfortunately $c_n$ grows with $n$ and the result is worse than the bounds from the 1-level density for small $r$ (which are the ones we care about most); see \cite{HM}.

Using the Paley-Wiener theorem to note the admissible test functions are the modulus squared of an entire function of exponential type 1 (or its Fourier transform as a convolution), Plancherel's theorem to convert to an equivalent minimization problem, and some Fredholm theory, in Appendix A of  \cite{ILS} the optimal test functions are computed for the 1-level density for the classical compact groups under the assumption that the support of the Fourier transform is contained in $[-2, 2]$. Our main result is to generalize these computations to larger support and higher $n$.

\begin{thm}\label{thm:mainresult} Let $\phi$ be an even Schwartz test function such that ${\rm supp}(\hphi) \subset [-2\sigma, 2\sigma]$, with $\sigma = s/2$. Then for $2 < s < 3$ the test function which minimizes the right hand side of \eqref{eq:formulatoboundrank} is given by  $\widehat{\phi} = g_{\mathcal{G},s} \ast \widecheck{g_{\mathcal{G},s}}$. Here $\ast$ represents convolution, $\widecheck{g_{\mathcal{G},s}}(x) = \overline{g_{\mathcal{G},s}}(-x)$, and $g_{\mathcal{G},s}$ is given by
\begin{equation}\label{SOEvenExplicitOptimalFunction}
    g_{{\rm SO(even)},s}(x) \ = \  \lambda_{{\rm SO(even)},s}
    \begin{cases}
      c_{1,\mathcal{G},s} \cos\left( \frac{|x|}{\sqrt{2}} \right) \ &|x| \le s/2 - 1 \\
      \cos\left( \frac{|x|}{2} - \frac{(\pi + 1)}{4} \right) \ &s/2 - 1 \le |x| \le 2 - s/2\\
      \frac{c_{1,\mathcal{G},s}}{\sqrt{2}} \sin\left( \frac{x-1}{\sqrt{2}} \right)  + c_{3,\mathcal{G},s} \ &2 - s/2 < |x| < s/2 \\
      0 &|x| \ge s/2,
  \end{cases}
      \end{equation}
      and
    \begin{equation}\label{OrthogOptimalFunction}
g_{{\rm O},s}(x) \ = \
\begin{cases}
 \frac{1}{1 + s/2} \ &|x| < s/2 \\
 0 \ &|x| \ge s/2
\end{cases}
    \end{equation}
for $\mathcal{G} \ = \  {\rm O}$, and
  \begin{equation}\label{SOOdd/SpExplicitOptimalFunction}
    g_{\mathcal{G},s}(x) \ = \  \lambda_{\mathcal{G},s}
    \begin{cases}
      c_{1,\mathcal{G},s} \cos\left( \frac{|x|}{\sqrt{2}} \right) \ &|x| \le s/2 - 1 \\
      \cos\left( \frac{|x|}{2} + \frac{(\pi - 1)}{4} \right) \ &s/2 - 1 \le |x| \le 2 - s/2\\
      \frac{-c_{1,\mathcal{G},s}}{\sqrt{2}} \sin\left( \frac{x-1}{\sqrt{2}} \right) + c_{3,\mathcal{G},s} \ &2 - s/2 < |x| < s/2 \\	
      0 & |x| \ge s/2
  \end{cases}
  \end{equation}
 for $\mathcal{G} =  {\rm SO(odd)}$ or ${\rm Sp}$. Here, the $c_{i,\mathcal{G}}$ and $\lambda_{\mathcal{G}}$ are easily explicitly computed, and are given later in \eqref{ActualCoefficientsSOEven}, \eqref{ActualCoeffSOOdd/Sp}, \eqref{SO(even)Scaling}, \eqref{SpScaling}, and \eqref{SOOddScaling}.
\end{thm}

\ \\
\begin{quote}
\texttt{For the rest of the paper, we let $\sigma = s/2$. Unless otherwise stated, $1 < \sigma < 1.5$, corresponding to the range for $s$. This notation is slightly at odds with other works in the literature, where the support of $\hphi$ is contained in $(-\sigma, \sigma)$ and for us it is $(-2\sigma, 2\sigma)$; we have elected to proceed this way as the natural object is $g$,  and the support of $\hphi$ is double that of $g$.}\end{quote}
\ \\

Moreover, the optimal function $g_{\mathcal{G},s}$, along with its coefficients $c_{i\mathcal{G},s}$ and its scaling factor $\lambda_{\mathcal{G},s}$, all depend on $s$. As this will be clear from equations \eqref{ActualCoefficientsSOEven} to \eqref{SOOddScaling}, to simplify the notation we often omit the subscript $s$ or $\sigma$,  as these are fixed in the analysis.

To help illustrate the main theorem, we include plots of the optimal $g$ for the groups SO(even), SO(odd), and Sp below in Figure \ref{fig:optplots}, and the plots for the corresponding optimal $\phi$ in Figure \ref{fig:optphiplots}; we do not include the optimal plots for the mixed orthogonal case, as the resulting $g$ is constant.

\begin{figure}[h]
\begin{center}
\scalebox{.5545}{\includegraphics{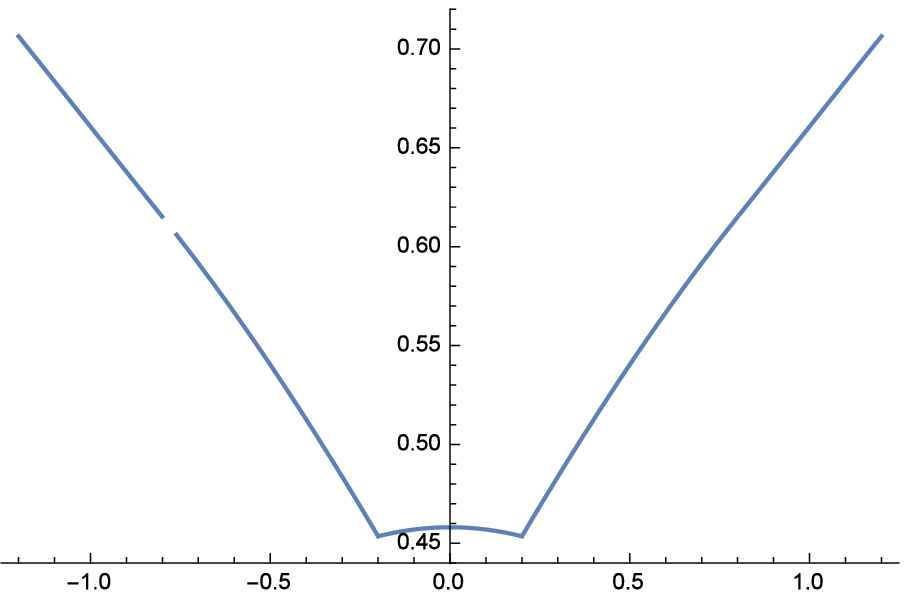}} \ \ \scalebox{.5545}{\includegraphics{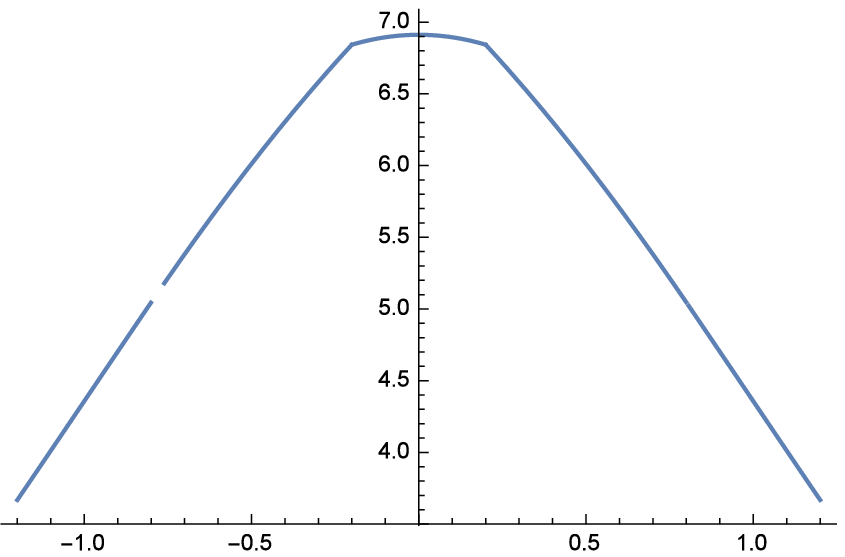}} \ \ \scalebox{.5545}{\includegraphics{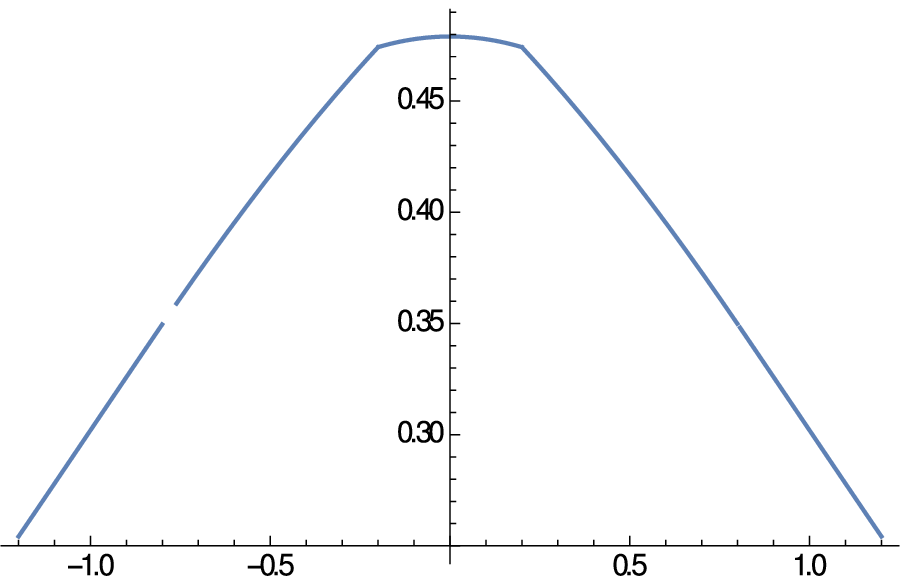}}
\caption{\label{fig:optplots} Plots of the optimal $g_{\mathcal{G}}$ with $\sigma = 1.2$ (and thus $s=2.4$). Left: Optimal SO(even) function. Middle: Optimal Sp function. Right: Optimal SO(odd) function. }
\end{center}
\end{figure}

\begin{figure}[h]
\begin{center}
\scalebox{.5545}{\includegraphics{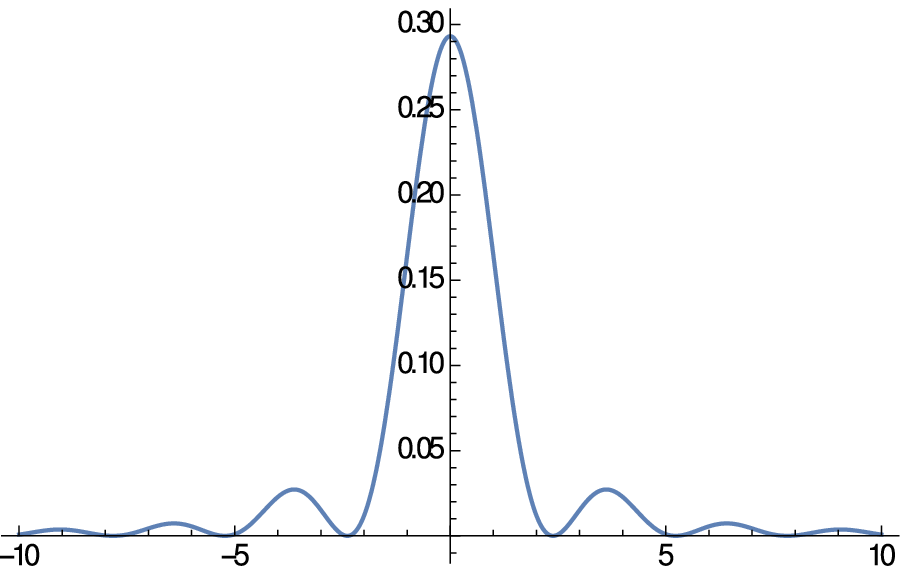}} \ \ \scalebox{.5545}{\includegraphics{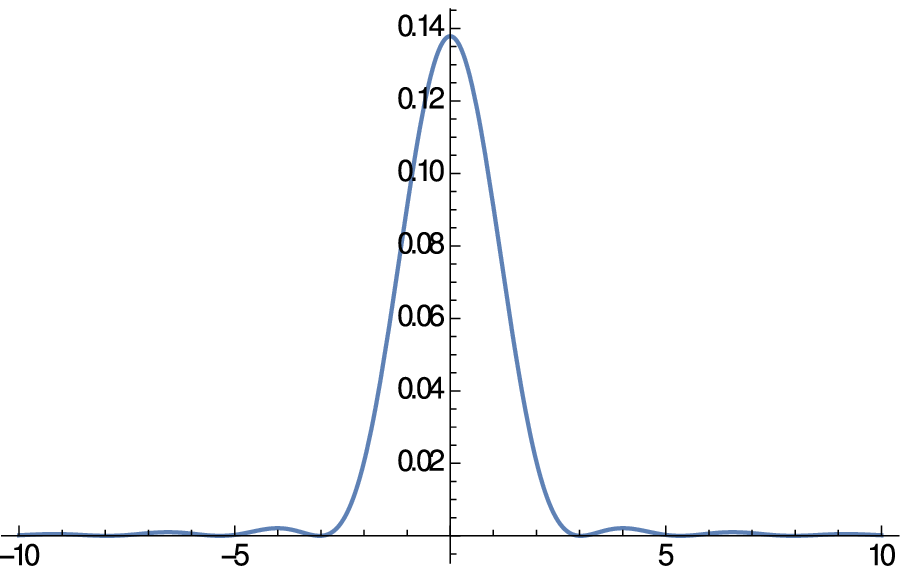}} \ \ \scalebox{.5545}{\includegraphics{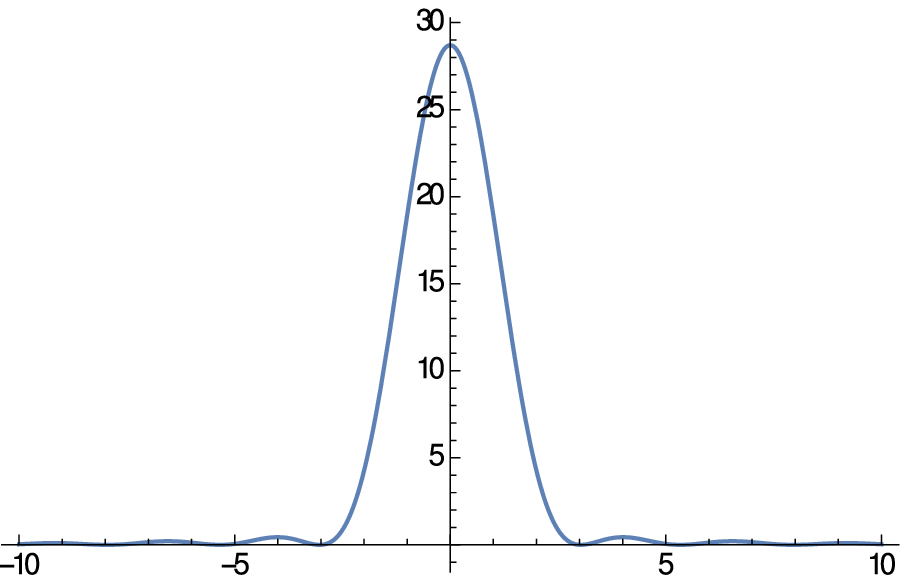}}
\caption{\label{fig:optphiplots} Plots of the optimal $\phi$ with $\sigma = 1.2$ (and thus $s=2.4$). Left: Optimal SO(even) function. Middle: Optimal Sp function. Right: Optimal SO(odd) function.  }
\end{center}
\end{figure}

As an immediate corollary we obtain the following bounds on the average rank. We isolate these upper bounds below. The record for largest support for the 1-level density are families of cuspidal newforms \cite{ILS} and Dirichlet $L$-functions \cite{FiM} (though see also \cite{AM} for Maass forms), where we can take $\sigma < 2$. It is possible to obtain better bounds on vanishing by using the 2 or higher level densities, though as remarked above in practice the reduced support means these results are not better than the 1-level for extra vanishing at the central point but do improve as we ask for more and more vanishing (see \cite{HM}).

\begin{cor}\label{cor:averankmainresult} Let $\mathcal{F}$ be a family of $L$-functions such that, in the limit as the conductors tend to infinity, the 1-level density is known to agree with the scaling limit of unitary, symplectic or orthogonal matrices. Then for every $\epsilon > 0$ in the limit the average rank is bounded above by
\begin{equation}
  \resizebox{.9\hsize}{!}{$
  \varepsilon +
  \begin{cases}
    \tiny  \frac{4 \sqrt{2} \sin \left(\frac{1}{4} (3-2 \sigma)\right)+2 (\sigma-1) \sin \left(\frac{1}{4} (-2 \sigma+\pi +3)\right)+\sin \left(\frac{1}{4} (2 \sigma+\pi -3)\right) \left(\sqrt{2} (\sigma+1) \tan \left(\frac{\sigma-1}{\sqrt{2}}\right)+2\right)}{8 \sqrt{2} \sin \left(\frac{1}{4} (3-2 \sigma)\right)+8 (\sigma-1) \sin \left(\frac{1}{4} (-2 \sigma+\pi +3)\right)+4 \sqrt{2} \sigma \sin \left(\frac{1}{4} (2 \sigma+\pi -3)\right) \tan \left(\frac{\sigma-1}{\sqrt{2}}\right)} \normalsize &\mathcal{G} \ = \ {\rm SO(even)} \\
    \tiny  \frac{-2 (\sigma-1) \sin \left(\frac{1}{4} (2 \sigma+\pi -3)\right)-4 \sqrt{2} \sin \left(\frac{1}{4} (3-2 \sigma)\right)+\sin \left(\frac{1}{4} (-2 \sigma+\pi +3)\right) \left(\sqrt{2} (\sigma-3) \tan \left(\frac{\sigma-1}{\sqrt{2}}\right)+2\right)}{8 (\sigma-1) \sin \left(\frac{1}{4} (2 \sigma+\pi -3)\right)+8 \sqrt{2} \sin \left(\frac{1}{4} (3-2 \sigma)\right)-4 \sqrt{2} (\sigma-2) \sin \left(\frac{1}{4} (-2 \sigma+\pi +3)\right) \tan \left(\frac{\sigma-1}{\sqrt{2}}\right)} \normalsize &\mathcal{G} \ = \ {\rm Sp} \\
    \tiny  \frac{6 (\sigma-1) \sin \left(\frac{1}{4} (2 \sigma+\pi -3)\right)+4 \sqrt{2} \sin \left(\frac{1}{4} (3-2 \sigma)\right)+\sin \left(\frac{1}{4} (-2 \sigma+\pi +3)\right) \left(\sqrt{2} (5-3 \sigma) \tan \left(\frac{\sigma-1}{\sqrt{2}}\right)+2\right)}{8 (\sigma-1) \sin \left(\frac{1}{4} (2 \sigma+\pi -3)\right)+8 \sqrt{2} \sin \left(\frac{1}{4} (3-2 \sigma)\right)-4 \sqrt{2} (\sigma-2) \sin \left(\frac{1}{4} (-2 \sigma+\pi +3)\right) \tan \left(\frac{\sigma-1}{\sqrt{2}}\right)} \normalsize &\mathcal{G} \ = \ {\rm SO(odd)} \\
    \frac{1}{2\sigma} + \frac{1}{2} &\mathcal{G} \ = \ {\rm O}.
  \end{cases}
  $
}
  \label{AvgRankBoundsForExtendedSupport}
\end{equation}
\end{cor}

\begin{remark}
We only list $g$ and not the optimal test functions or their Fourier transforms above, as we do not need either function for the computation of the infimum. \cite{ILS} show that given the $g$ associated to the optimal function, the infimum is given by
\begin{equation}\label{infcomputationeqn}
  {\rm inf}(\mathcal{G},\sigma)\ =\ \frac{1}{\int_{-\sigma}^{\sigma} g(x) dx},
\end{equation}
where the integral above exists and is nonzero by \eqref{detailsofhatfactorization} and  \eqref{CriterionForOptimal}, both established later.
\end{remark}

A natural choice for a test function is the Fourier pair
\begin{equation}
\phi(x) = \left( \frac{\sin(2 \sigma \pi x)}{2 \sigma \pi x} \right)^{2}, \quad \quad \widehat{\phi}(y) = \frac{1}{2\sigma}\left( 1 - \frac{|y|}{2\sigma} \right) \ \quad \textrm{if} \ |y| < 2 \sigma;
\label{NaiveFourierPair}
\end{equation}
this is the function used for the initial computation of average rank bounds in \cite{ILS} and are optimal for $\sigma = 1$. For the groups ${\rm SO(even)}, {\rm Sp}, {\rm SO(odd)}$, and for $1 < \sigma < 1.5$ the functions we find provide a significant improvement for the upper bounds on average rank over the pair \eqref{NaiveFourierPair}. We illustrate the improvement in Figure \ref{fig:comparisons}, which is much easier to process than \eqref{AvgRankBoundsForExtendedSupport}.

\begin{figure}[h]
\begin{center}
\scalebox{.5545}{\includegraphics{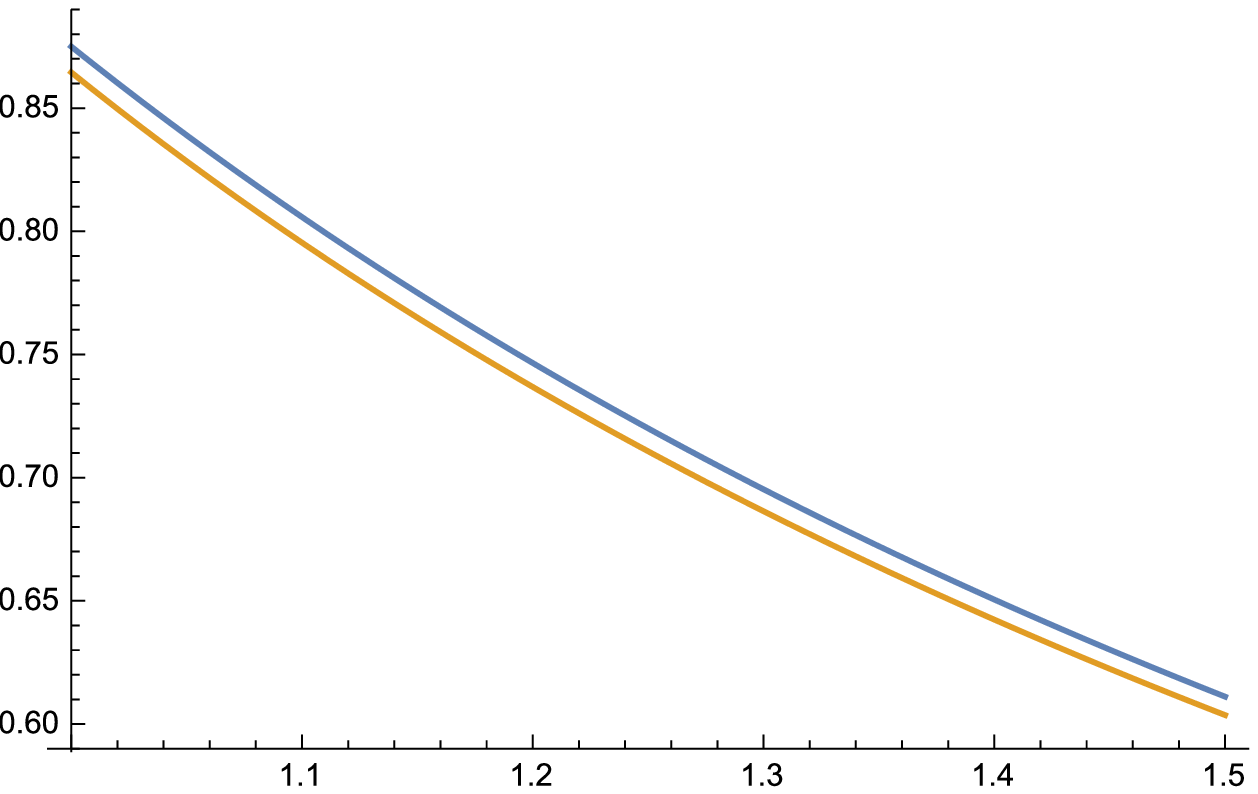}} \ \ \scalebox{.5545}{\includegraphics{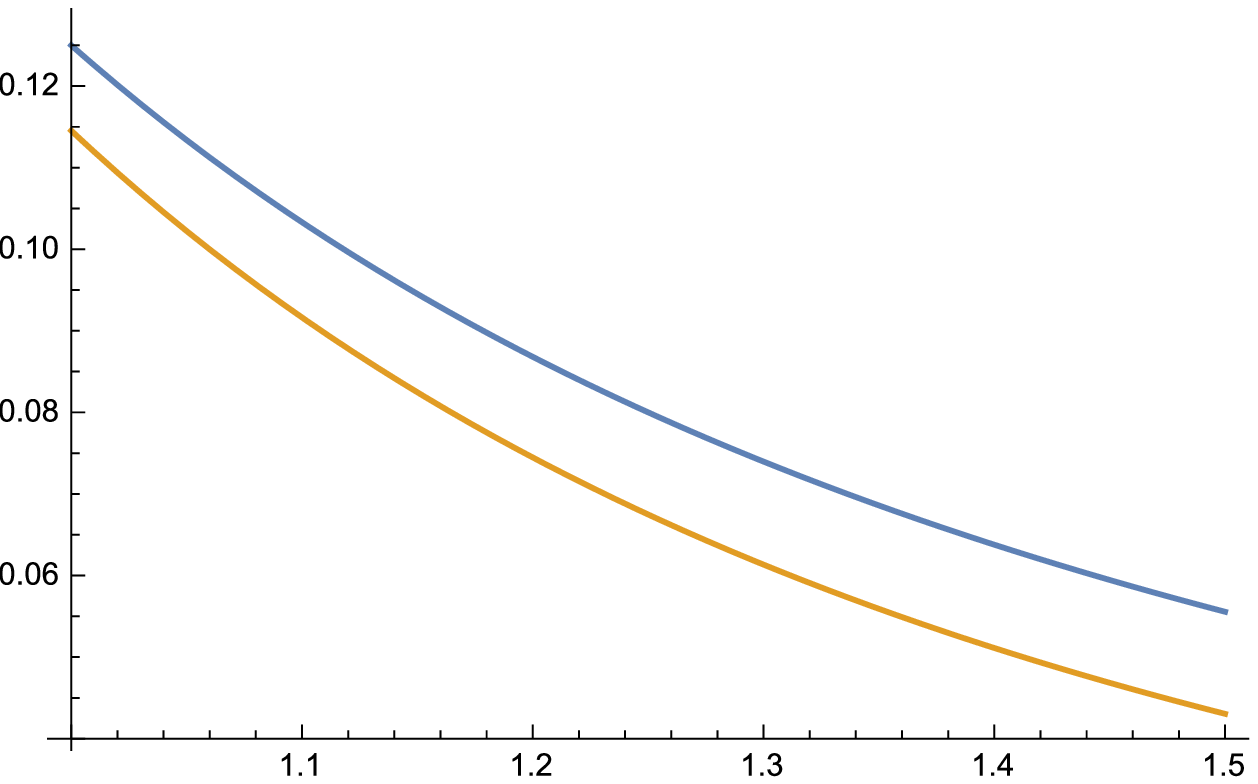}} \ \ \scalebox{.5545}{\includegraphics{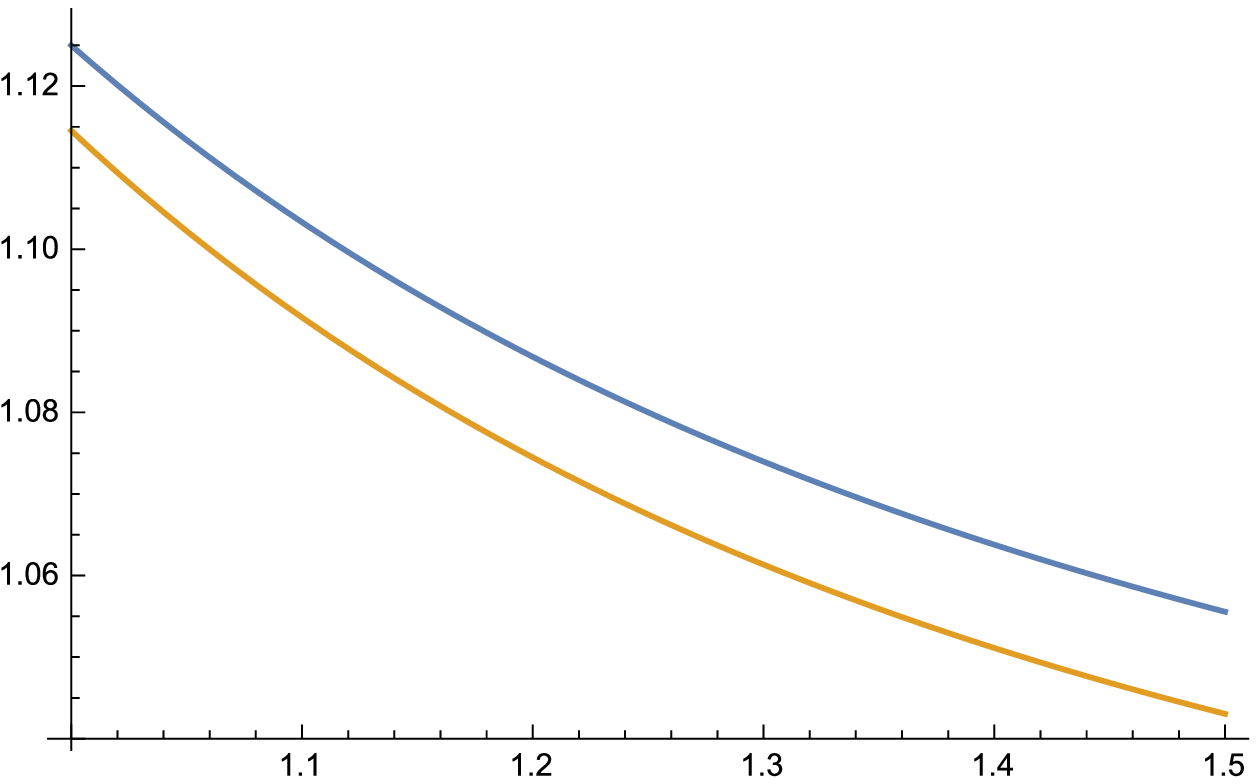}}
\caption{\label{fig:comparisons} Comparison of upper bounds. The larger bound is from using the sub-optimal naive guess \eqref{NaiveFourierPair}, the lower is from using our results from \eqref{AvgRankBoundsForExtendedSupport}. Left: $\mathcal{G}$ = SO(even). Middle: $\mathcal{G} = {\rm Sp}$. Right: $\mathcal{G}$ = SO(odd).}
\end{center}
\end{figure}





\subsection{Sketch of Proof}

The first step in our proof is to note that it follows from the Paley-Wiener theorem and Ahiezer's theorem that the admissible functions $\phi$, with ${\rm supp}(\widehat{\phi}) \subset [-2\sigma,2\sigma]$ satisfy
\begin{equation}\label{optimalphihatfactors}
\widehat{\phi}(y)\ =\ (g \ast \check{g})(y),
\end{equation}
where
\begin{equation} \label{detailsofhatfactorization}
\check{g}(y) = \overline{g}(-y) \quad \quad {\rm supp}(g) \subset [-\sigma,\sigma], \ g \in L^{2}[-\sigma,\sigma];
\end{equation}
see Appendix A of \cite{ILS}. We will sometimes refer to an ``optimal $g$''. By this, we mean the $g$ that satisfies \eqref{optimalphihatfactors} and \eqref{detailsofhatfactorization} for the optimal $\widehat{\phi}$ at a fixed level of support.

%

The broad strategy of the proof of Theorem \ref{thm:mainresult} is to use an operator equation from \cite{ILS} to show (non-constructively) that for all $\sigma \in \R^{+}$, there exists a unique optimal test function with ${\rm supp}(\widehat{\phi}) \subseteq [-2\sigma,2\sigma]$ that minimizes the functional
\begin{equation}\label{OptimalityCriterion}
\frac{\int_{-\infty}^{\infty}\phi(x)W_{\mathcal{G}}(x) dx}{\phi(0)}.
\end{equation}
We then find a collection of necessary conditions that leave us with precisely one choice for $\phi$.

More explicitly, our argument proceeds as follows.

\begin{enumerate}

  \item \label{firstgoal} We show that certain optimality criterion on $\widehat{\phi}$ presented in \cite{ILS} holds for all $\sigma \in \R^{+}$ (here ${\rm supp}(\widehat{\phi}) \subset [-2\sigma,2\sigma])$.

  \item \label{secondgoal} We show that $g$ is smooth almost everywhere, where $g \ast \check{g} = \widehat{\phi}$ and ${\rm supp}(g) \subseteq
      (-\sigma,\sigma)$.

  \item \label{thirdgoal} Our kernels give us a series of location-specific integral equations. Using the previous smoothness result, we convert those to a
      system of location-specific delay differential equations.

  \item \label{fourthgoal} We solve this system to find an $n$-parameter family in which our solution lives. To find this solution, we incorporate symmetries of       $g$ -- namely that it is even.

  \item \label{fifthgoal} Incorporating more necessary conditions on $g$, we reduce the family to a single candidate function -- by our existence result this is our $g$, from which we obtain our optimal test function $\phi$.

\end{enumerate}

From the list above, we will accomplish goal \ref{firstgoal} in \S\ref{firstgoalsec}, goal \ref{secondgoal} in \S\ref{secondgoalsec}, goal \ref{thirdgoal} in \S\ref{thirdgoalsec}, goal \ref{fourthgoal} in \S\ref{fourthgoalsec}, and goal \ref{fifthgoal} in \S\ref{fifthgoalsec}. The proof of the optimal functions for $\mathcal{G} = {\rm O}$ is significantly easier than the proofs for the other functions. We include a brief proof of this fact at the end of \S\ref{secondgoalsec}. Finally, we conclude with some remarks about how these results are used in number theory, and discuss ongoing and future research.

\section{Extension of the Conditions of \cite{ILS}}\label{firstgoalsec}

Our first step is to state and extend an optimality criterion on $g_{\mathcal{G}}$, analogous to that in Appendix A of \cite{ILS} (we will state it in \eqref{CriterionForOptimal}). Following their arguments, we seek to minimize the functional
\begin{equation}
  R(g) \ := \  \frac{\innerproduct{(I + K_{\mathcal{G}}) g,g}}{|\innerproduct{g,1}|^{2}},
  \label{FunctionalOnTransformSide}
\end{equation}
where $I$ is the identity operator,
\begin{equation}
  K_{\mathcal{G},\sigma}g(x) \ = \  \int_{-\sigma}^{\sigma} m_{\mathcal{G}}(x - y)g(y) dy,
  \label{FormOfK}
\end{equation}
and
\begin{align}
  m({\rm SO(even)})(\xi) &\ = \  \frac{1}{2}\textrm{I}_{[-1,1]}(\xi) \notag \\
  m({\rm SO(odd)})(\xi) &\ = \  1 - \frac{1}{2}\textrm{I}_{[-1,1]}(\xi) \notag \\
  m({\rm Sp}) (\xi) &\ = \  -\frac{1}{2}\textrm{I}_{[-1,1]}(\xi) \notag \\
  m({\rm O})(\xi) &\ = \  \frac{1}{2}
  \label{GroupKernels}
\end{align}
where ${\rm I}_{[-1,1]}$ is the indicator function for the interval $[-1,1]$.

\begin{lem}
  The operator $K_{\mathcal{G},\sigma}$ is compact for all $\sigma \in \R^{+}$, and all choices of $\mathcal{G}$.
  \label{KCompact}
\end{lem}

\begin{proof} As the functions in \eqref{GroupKernels} are all clearly in $L^{2}([-\sigma,\sigma])$, they are trace class and therefore compact. \end{proof}

It follows that the operator $(I + K_{\mathcal{G},\sigma})$ satisfies the Fredholm alternative for all $\mathcal{G}$ and all $\sigma \in \R^{+}$. Applying the arguments from \cite{ILS} shows that for all $\sigma$ the operator $I + K_{\mathcal{G}, \sigma}$ is still positive definite. Thus there is some $g$ such that
\begin{equation}
  (I + K_{\mathcal{G},\sigma})(g) \ = \  1.
  \label{CriterionForOptimal}
\end{equation}
Again, following the arguments of \cite{ILS}, one can show that this $g$ indeed minimizes \eqref{FunctionalOnTransformSide}. This completes the first step.

We are now ready to find the optimal functions for $\mathcal{G} = {\rm O}$.

\begin{lemma}\label{optimalorthogfunctionslemma}
For $\mathcal{G} = {\rm O}$, and for any $\sigma \in \R^{+}$, the optimal test function for the minimization of \eqref{OptimalityCriterion} is
\begin{equation}\label{optimalorthogfunctionseqn}
  \phi(x) \ = \ \left( \frac{\sin (2 \sigma \pi x)}{(1 + \sigma) \pi x} \right)^{2}
\end{equation}
and the associated upper bound on average rank is
\begin{equation}
  \frac{1}{2\sigma} + \frac{1}{2}
  \label{rankboundorthog}
\end{equation}
\end{lemma}

\begin{proof}[Proof -- Optimal functions for $\mathcal{G} = {\rm O}$]
Using the criterion \eqref{CriterionForOptimal}, we can find the optimal functions for $\mathcal{G} = {\rm O}$ for all $\sigma \in \R^{+}$. Trying constant functions, with $K_{\mathcal{G}} = \frac{1}{2}{\rm I}_{[-\sigma,\sigma]}$, we see that
\begin{equation}\label{optimalorthogfunctions}
  g(x) \ = \
  \begin{cases}
    \frac{1}{1 + \sigma} \ &|x|\le \sigma \\
    0 \ &|x| > \sigma
  \end{cases}
\end{equation}
satisfies \eqref{CriterionForOptimal}. By \eqref{optimalphihatfactors} and \eqref{detailsofhatfactorization}, we know $\widehat{\phi} \ = \ g \ast g$. Thus $\phi \ = \ (\mathcal{F}^{-1}(g))^{2}$, where $\mathcal{F}^{-1}$ denotes Fourier inversion. Two quick calculations yield \eqref{optimalorthogfunctionseqn} and \eqref{rankboundorthog}.
\end{proof}

\section{Smoothness Almost Everywhere}\label{secondgoalsec}

We now show that for an optimal $\phi$ such that $\widehat{\phi} = g \ast \check{g}$, $g$ must be Lipschitz continuous. Then we show that such a function is differentiable almost everywhere, using a theorem of Rademacher.

First, we show that $g$ is bounded.

\begin{lem}\label{BoundedBeforeLip} Let $\widehat{\phi}$ Fourier transform of the optimal function, supported in $[-2\sigma,2\sigma]$, then $g$ (in the sense of \eqref{optimalphihatfactors} and \eqref{detailsofhatfactorization}) is bounded.
\end{lem}

\begin{proof}  We show that
\begin{equation}\label{FirstPartTriangle}
h(x)\ := \ \int_{-\sigma}^{\sigma} m_{\mathcal{G}}(x - y)g(y) dy
\end{equation}
is bounded. We know that $g \in L^2([-\sigma,\sigma])$. By the Cauchy-Schwarz inequality, we have
\begin{equation}\label{CauchySchwarzBounded}
    \left|\int_{-\sigma}^{\sigma} m_{\mathcal{G}}(x - y) g(y) dy\right| \ \le\ \norm{g}_{L^2} \norm{m_{\mathcal{G}}(x - y)}_{L^2}\ <\ \infty,
\end{equation} and thus $h$ is bounded.

By \eqref{CriterionForOptimal}, we know that for the optimal $g$ we have $g + h = 1$. As $h$ is bounded the optimal $g$ must bounded.
\end{proof}

\begin{prop}\label{LipschitzProp} For any $\sigma \in \R^{+}$, the optimal $g$ is Lipschitz continuous.
\end{prop}

\begin{proof}
  Using \eqref{CriterionForOptimal} and applying the maximum modulus inequality, we see that for $x_1,x_2 \in [-\sigma,\sigma]$,
  \begin{align}\label{LipschitzArg}
    \left| g(x_1) - g(x_2) \right| &\ = \  \int_{-\sigma}^{\sigma}(m_{\mathcal{G}}(x_1 - y) - m_{\mathcal{G}}(x_2 - y))g(y) dy \notag \\
    &\ \le\ \int_{-\sigma}^{\sigma} |m_{\mathcal{G}}(x_1 - y) - m_{\mathcal{G}}(x_2 - y)| |g(y)| dy \notag \\
    &\ \le\ \max_{y \in [-\sigma,\sigma]}|g(y)| \int_{-\sigma}^{\sigma} |m_{\mathcal{G}}(x_1 - y) - m_{\mathcal{G}}(x_2 - y)| dy.
      \end{align}

We now analyze \eqref{LipschitzArg}. Notice that for all choices of $m_{\mathcal{G}}$  in \eqref{GroupKernels}, the integrand is bounded by 1/2. We will examine the region of integration. Without loss of generality we may assume $x_1 \ge x_2$. Note that our integrand vanishes everywhere except from $\max\left\{x_1 - 1, x_2 + 1\right\}$ to $\min\left\{x_1 + 1, \sigma\right\}$ and again from $\max\left\{-\sigma, x_2 - 1\right\}$ to $\min\left\{x_1 - 1, x_2 + 1\right\}$. Thus the region of integration has measure at most $\min\left\{2(x_1 - x_2), 4\right\}$, and the integrand vanishes outside of a set of measure at most $2(x_1 - x_2)$.

We may now revise the inequality in \eqref{LipschitzArg}:
  \begin{align}\label{CompletedLipschitzArg}
    \left| g(x_1) - g(x_2) \right|  &\ \le\ \max_{y \in [-\sigma,\sigma]}|g(y)| \int_{-\sigma}^{\sigma} |m_{\mathcal{G}}(x_1 - y) - m_{\mathcal{G}}(x_2 - y)| dy \notag\\
    &\ \le\ \max_{y \in [-\sigma,\sigma]}|g(y)| (2|x_1 - x_2|) \left( \frac{1}{2} \right) \notag \\
    &\ \le\ \max_{y \in [-\sigma,\sigma]} |g(y)| |x_1 - x_2|,
      \end{align} completing the proof that $g(x)$ is a Lipschitz continuous function.
\end{proof}

We use a theorem of Rademacher to show that our function $g$ is differentiable almost everywhere.

\begin{thm}[Rademacher (see Theorem 3.1.6 of \cite{Fed})]\label{RademacherTheorem}
  Let $\Omega \subset \R^n$ be open. If $f: \Omega \longrightarrow \R$ is Lipschitz continuous, then $f$ is differentiable almost everywhere in $\Omega$.
\end{thm}

We immediately obtain the following.

\begin{cor} For all $\sigma \in \R^{+}$, the optimal $g$ is differentiable almost everywhere.
  \label{gisdiffa.e.}
\end{cor}

\begin{proof}
  Let $\widetilde{g}(x)$ be $g(x)$ restricted to $(-\sigma,\sigma)$. The result for $\widetilde{g}(x)$ follows from Proposition \ref{LipschitzProp} and Theorem \ref{RademacherTheorem}. Thus $g$ is differentiable almost everywhere in $(-\sigma,\sigma)$, which is almost everywhere in $[-\sigma,\sigma]$.
\end{proof}

Finally, we show that each such $g$ is in fact infinitely differentiable almost everywhere.

\begin{lem}\label{TwiceDiff} The optimal $g$ is infinitely differentiable almost everywhere.
\end{lem}

\begin{proof}
We proceed by induction. Our base case, that $g$ is once-differentiable, is established by Corollary \ref{gisdiffa.e.}. For the inductive step, we assume that $g$ is $k$-times differentiable.

Note that for any choice of $m$, we have
\begin{equation}\label{GrossSimplification}
  g(x) \ = \  1 - \left( \alpha_{1,\mathcal{G}} \int_{a}^{b} g(y) dy + \alpha_{2,\mathcal{G}} \int_{c_1}^{f_1(x,\sigma)} g(y) dy + \alpha_{3,\mathcal{G}} \int_{c_2}^{f_2(x,\sigma)} g(y) dy
   \right),
\end{equation}
where the $\alpha_i$ are scalars and $f_i(x,\sigma)$ is either a constant or a smooth function of $x$. In particular, for $\sigma > 1$ at least one of the $f_i$ is a smooth function of $x$. We know that $g$ is continuous. By the fundamental theorem of calculus and the chain rule, the expression on the righthand side of \eqref{GrossSimplification} is $k+1$ times differentiable. This completes the proof of the inductive step, and thus $g$ is smooth almost everywhere.
\end{proof}

\section{A System of Integral Equations}

To establish our integral equations, we first show that the optimal $g$ is even.

\begin{lem}\label{EvenLemma}
  The optimal $g$ is even.
\end{lem}

\begin{proof}
  The key to this proof is that any choice of $m$ from \eqref{GroupKernels} is even. We show that $g(-x)$ also satisfies \eqref{CriterionForOptimal} and so must be equal to $g$. As $m$ is even, we have
  \begin{align}
    g(-x) + \int_{-\sigma}^{\sigma} m(x - y)g(-y) dy &\ = \  g(-x) + \int_{-\sigma}^{\sigma} m(x + y) g(y) dy \nonumber\\
    &\ = \  g(-x) + \int_{-\sigma}^{\sigma} m(-x - y)g(y) dy \nonumber\\
    &\ = \  (I + K_{\mathcal{G},\sigma})(g)(-x),
  \end{align}
  which is equal to one. By uniqueness, $g(x) = g(-x)$ and $g$ is even.
\end{proof}

By the results of the previous section, finding the optimal $\widehat{\phi}$ for $2 < 2\sigma < 3$ involves finding the optimal $g$ for $1 < \sigma < 1.5$. We claim, (momentarily without justification) that there are three intervals of importance in our study of this function. These are
\begin{align}
  J_1 &\ := \  [0, \sigma - 1] \nonumber\\
  J_{2} &\ := \  [\sigma - 1,2 - \sigma] \nonumber\\
  J_{3} &\ := \  [2 - \sigma, \sigma].
\label{IntervalDefs}
\end{align}
As $g$ is even, it suffices to find $g$ on $[0,\sigma]$, which means finding $g$ on all of the intervals above. Examining the kernels in \eqref{GroupKernels} and the requirement \eqref{CriterionForOptimal}, we see that for $x \in J_{1}$, the optimal $g$ satisfies
\begin{equation}
  g(x) + \alpha_{2,\mathcal{G}} \int_{0}^{x + 1} g(y) dy + \alpha_{2,\mathcal{G}}\int_{0}^{1-x} g(y) dy + \alpha_{1,\mathcal{G}} \int_{0}^{\sigma} g(y) dy \ = \  1,
  \label{IntEqnOnI1}
\end{equation}
and for $x \in J_{2}$ or $J_{3}$, we have
\begin{equation}
  g(x) + \alpha_{2,\mathcal{G}} \int_{x - 1}^{\sigma} g(y) dy + \alpha_{1,\mathcal{G}}\int_{0}^{\sigma} g(y) dy  \ = \  1.
  \label{IntEqnOnI23}
\end{equation}
In equations \eqref{IntEqnOnI1} and \eqref{IntEqnOnI23}, we note that $\alpha_{1,\mathcal{G}} = 0$ for $\mathcal{G} \not = {\rm SO(odd)}$ and $1$ for $\mathcal{G} = {\rm SO(odd)}$ and $\alpha_{2,\mathcal{G}} = 1/2$ for $\mathcal{G} = {\rm SO(even)}$ and $-1/2$ for $\mathcal{G} = {\rm Sp}$ or $\mathcal{G} = {\rm SO(odd)}$.

\subsection{Conversion to Location-Specific System of Delay Differential Equations}\label{thirdgoalsec}
Lemma \ref{TwiceDiff} justifies differentiation of \eqref{IntEqnOnI1} and \eqref{IntEqnOnI23} under the integral signs, which gives the following system of location-specific delay differential equations:
\begin{align}
  g'(x) + \alpha_{2,\mathcal{G}}g(x + 1) - \alpha_{2,\mathcal{G}}g(1-x) \ = \  0
  \label{DelayDiff1}\\
  g'(x + 1) - \alpha_{2,\mathcal{G}}g(x) \ = \  0,
  \label{DelayDiff2}
\end{align}
where \eqref{DelayDiff1} holds for $x \in J_{1}$, and \eqref{DelayDiff2} holds for $x+1 \in J_{2}$ or $J_{3}$.

\subsection{Solving The System}\label{fourthgoalsec}


\begin{lem} The optimal $g$ satisfies
  \begin{equation}\label{OptimalgOnTwoIntervalsLemma} \twocase{g(x) \ = \ }{c_{1} \cos\left( \frac{x}{\sqrt{2}} \right) + c_{2} \sin \left( \frac{x}{\sqrt{2}} \right)}{if $x \in J_{1}$}{c_{1}\alpha_{2,\mathcal{G}}\sqrt{2} \sin\left( \frac{x - 1}{\sqrt{2}} \right) - c_{2}\alpha_{2,\mathcal{G}} \sqrt{2} \cos\left( \frac{x-1}{\sqrt{2}} \right) + c_{3}}{if $x \in J_{3}$}
  \end{equation}
  for some $c_{i} \in \R$.
\end{lem}

Before proving this lemma, it is important to note the following symmetry among our intervals. We first set some  notation. If $a$ is a number and $I$ is an interval, \be a-I \ := \ \{x: x = a - y, y \in I\}. \ee

Note that $J_{j}$ is one of the intervals defined in
\begin{equation}
  1 - J_{1}\ \subseteq\ J_{3}.
  \label{SymmetryI1I3}
\end{equation}
We also mention that
\begin{equation}
  1 - J_{2} \ = \  J_{2},
  \label{SymmetryI2I2}
\end{equation}
though we will not use this fact until later.

\begin{proof}
Differentiating \eqref{IntEqnOnI1} yields
  \begin{equation}
    g''(x) + \alpha_{2,\mathcal{G}}g'(x+1) + \alpha_{2,\mathcal{G}}g'(1-x) \ = \  0.
    \label{DiffFirstEq}
  \end{equation}
Because of the symmetry \eqref{SymmetryI1I3}, we may use equation \eqref{IntEqnOnI23} on both the $x+1$ and $1-x$ terms. This gives us the following equation:
  \begin{equation}\label{ApplyingSecondEquation}
    g''(x) + \alpha_{2,\mathcal{G}}^{2}(g(x) + g(-x)) \ = \ 0.
  \end{equation}

By Lemma \ref{EvenLemma}, $g$ is even. Moreover, $\alpha_{2,\mathcal{G}} = \pm 1/2$. So, for any group $\mathcal{G}$, \eqref{ApplyingSecondEquation} simplifies to
  \begin{equation}\label{FinalDiffEq-33Case}
    g''(x) + \frac{1}{2} g(x) \ = \  0,
     \end{equation}
which is a standard differential equation and easily solved. Its solution, which applies to $g$ on $J_{1}$, is $c_{1}\cos(x/\sqrt{2}) + c_{2} \sin(x/\sqrt{2})$ for some constants $c_1, c_2$. We find the three parameter family for $g$ on $J_{3}$ by applying \eqref{DelayDiff2} to this result.
\end{proof}

\ \\
Note that because of the symmetry \eqref{SymmetryI2I2}, the associated delay differential equation on interval two is different. It is
\begin{equation}
  g'(x) - \alpha_{2,\mathcal{G}}g(1-x) \ = \  0.
  \label{DiffEqVersion}
\end{equation}

\begin{lem}
  The delay differential equation \eqref{DiffEqVersion} has a unique one-parameter family of solutions in the class $C^2$. That family is
  \begin{equation}
    c_{1}\cos\left( \alpha_{2,\mathcal{G}}x - \left( \frac{\pi + 2\alpha_{2,\mathcal{G}}}{4} \right) \right)
    \label{OneParamFamEqnMiddlePart}
  \end{equation}
  with $c_1 \in \R$.
  \label{UniqueOneParameterFamily}
\end{lem}

\begin{proof}
  Differentiate \eqref{DiffEqVersion} to obtain
  \begin{align}
    f_{\mathcal{G}}''(x) \ = \  -\alpha_{2,\mathcal{G}}f_{\mathcal{G}}'(1-x) \ = \  -\alpha_{2,\mathcal{G}}^{2}f_{\mathcal{G}}(x) \ = \  -\frac{1}{4}f_{\mathcal{G}}(x),   \label{SecondOrderDiffForf}
  \end{align}
  where we obtain the second equality by applying \eqref{DiffEqVersion} to $f_{\mathcal{G}}'(1-x)$. The third equality is simply a subsitution for $\alpha_{2,\mathcal{G}}$. However, equation \eqref{SecondOrderDiffForf} is a standard linear differential equation that has a two-parameter family of solutions given by
  \begin{equation}
    c_1 \cos \left( \frac{x}{2} \right) + c_2 \sin \left( \frac{x}{2}  \right).
    \label{TwoParameterFamily}
  \end{equation}
  We now apply \eqref{DiffEqVersion} to narrow this family down to a one-parameter family. The differential equation \eqref{DiffEqVersion} and trigonometric angle addition formulae yield the relation
  \begin{align}\label{NoNumber,ButTheLinAlgEqn}
    \frac{1}{2}\left( -c_1\sin\left( \frac{x}{2} \right) + c_2 \cos \left( \frac{x}{2} \right) \right) &\ = \   \alpha_{2,\mathcal{G}} \left( c_1 \cos\left( \frac{1}{2}\right) + c_{2}\sin\left( \frac{1}{2} \right)  \right) \cos\left( \frac{x}{2}\right) \nonumber\\
    &\ \ \ \ \ +\ \alpha_{2,\mathcal{G}} \left( c_1 \sin\left( \frac{1}{2} \right) - c_{2} \cos\left( \frac{1}{2} \right) \right)\sin\left( \frac{x}{2} \right).
  \end{align}
  In order for the expression above to vanish, we need the coefficients on $\cos(x/2)$ and $\sin(x/2)$ to both be zero. This translates into the requirement that the vector
  $\begin{pmatrix}
    c_1 \\
    c_2
  \end{pmatrix}$ be in the nullspace of the matrix
  \begin{equation}
      \begin{pmatrix}
	2 \alpha_{2,\mathcal{G}} \cos(1/2) & 2\alpha_{2,\mathcal{G}}\sin(1/2) - 1 \\
    2\alpha_{2,\mathcal{G}}\sin(1/2) + 1 & -2\alpha_{2,\mathcal{G}}\cos(1/2)
  \end{pmatrix}.
  \label{NullspaceRequirement}
\end{equation}
Note the matrix in \eqref{NullspaceRequirement} has determinant
\begin{equation}
  -4\alpha_{2,\mathcal{G}}^2(\sin^2(1/2) + \cos^2(1/2)) + 1 \ = \  0
  \label{ZeroDeterminant}
\end{equation}
because $\alpha_{2,\mathcal{G}} = \pm 1/2$. We know from \cite{ILS} that for each $\mathcal{G}$ the function in \eqref{OneParamFamEqnMiddlePart} is a solution to \eqref{DiffEqVersion}.  From our determinant argument, we know all solutions to that differential equation are all scalar multiples of a single nonzero solution, completing the proof.
\end{proof}

\section{Finding Coefficients}\label{fifthgoalsec}
Substituting values for $\alpha_{i,\mathcal{G}}$ for $i=1,2$, we find
\begin{equation}\label{SOEvenExplicitOptimalFunction2}\resizebox{.9\hsize}{!}{$
    g_{{\rm SO(even)}}(x) \ = \  \lambda_{{\rm SO(even)}}
    \begin{cases}
      c_{1,\mathcal{G}} \cos\left( \frac{|x|}{\sqrt{2}} \right) + c_{2,\mathcal{G}} \sin\left( \frac{|x|}{\sqrt{2}} \right) \ &|x| \ \le\ \sigma - 1 \\
      \cos\left( \frac{|x|}{2} - \frac{(\pi + 1)}{4} \right) \ &\sigma - 1\ \le\ |x|\ \le\ 2 - \sigma\\
      \frac{c_{1,\mathcal{G}}}{\sqrt{2}} \sin\left( \frac{x-1}{\sqrt{2}} \right) - \frac{c_{2,\mathcal{G}}}{\sqrt{2}} \cos\left( \frac{x-1}{\sqrt{2}} \right) + c_{3} \ &2 - \sigma\ <\ |x|\ \le\ \sigma	
    \end{cases}
  $}
      \end{equation}
  for $\mathcal{G} = {\rm SO(even)}$ and
  \begin{equation}\label{SOOdd/SpExplicitOptimalFunction2}
    g_{\mathcal{G}}(x) \ = \  \lambda_{\mathcal{G}}
    \begin{cases}
      c_{1,\mathcal{G}} \cos\left( \frac{|x|}{\sqrt{2}} \right) + c_{2,\mathcal{G}} \sin\left( \frac{|x|}{\sqrt{2}} \right)\ &|x|\ \le\ \sigma - 1 \\
      \cos\left( \frac{|x|}{2} + \frac{(\pi - 1)}{4} \right) \ &\sigma - 1\ \le\ |x|\ \le\ 2 - \sigma\\
      \frac{-c_{1,\mathcal{G}}}{\sqrt{2}} \sin\left( \frac{x-1}{\sqrt{2}} \right) + \frac{c_{2,\mathcal{G}}}{\sqrt{2}} \cos\left( \frac{x-1}{\sqrt{2}} \right) + c_{3}\ &2 - \sigma\ <\ |x|\ \le\ \sigma	
    \end{cases}
      \end{equation}
  for $\mathcal{G} = {\rm SO(odd)}$ or Sp.

  \begin{lem}\label{OptimalCoeffExist}
    There exist unique, computable coefficients $c_{i, \mathcal{G}}, \lambda_{\mathcal{G}}$ (for $i = 1,2,3$) so that the functions \eqref{SOEvenExplicitOptimalFunction2} and \eqref{SOOdd/SpExplicitOptimalFunction2} satisfy $(I + K_{\mathcal{G}, \sigma}) = 1$ and are thus optimal.
  \end{lem}

  \begin{proof}
    We use \eqref{CriterionForOptimal} and Lemma \ref{TwiceDiff} to find more necessary conditions on such a $g_{\mathcal{G}}$. In particular, we impose the three relations:
  \begin{align}
    \lim_{x \to (\sigma - 1)^{-}}g_{1,\mathcal{G}}(x) &\ = \  \lim_{x \to (\sigma - 1)^{+}}g_{\mathcal{G}}(x) \label{Req1:LeftContinuity}\\
    (I + K_{\mathcal{G}})(g_{\mathcal{G}})(0) &\ = \  (I + K_{\mathcal{G}})(g_{\mathcal{G}})(.5) \label{Req2:LeftAndMidScalarsAgree}\\
(I + K_{\mathcal{G}})(g_{\mathcal{G}})(0) &\ = \  (I + K_{\mathcal{G}})(g_{\mathcal{G}})(\sigma).
    \label{Req3:LeftAndRightScalarsAgree}
  \end{align}
  The first gives continuity, the second and third ensure that $(I + K_{\mathcal{G},\sigma})$ is constant; however, they do not ensure it is 1. That is accomplished by using $\lambda_{\mathcal{G}}$ to appropriately scale down the function. This gives us the matrix equations
 \begin{equation}
   \resizebox{.9\hsize}{!}{$
  \begin{pmatrix}
    \cos\left( \frac{\sigma - 1}{\sqrt{2}} \right) & \sin\left( \frac{\sigma - 1}{\sqrt{2}} \right) & 0 \\
    \cos\left( \frac{\sigma - 1}{\sqrt{2}} \right) & 0 & 0 \\
 \frac{1}{\sqrt{2}}\sin\left( \frac{\sigma - 1}{\sqrt{2}} \right) + \cos\left( \frac{\sigma - 1}{\sqrt{2}} \right) & \sqrt{2} - \frac{1}{\sqrt{2}}\cos\left( \frac{\sigma -1}{\sqrt{2}} \right) & -1
  \end{pmatrix}
  \begin{pmatrix}
    c_1 \\
    c_2 \\
    c_3
  \end{pmatrix}
  \ = \
  \begin{pmatrix}
    g_{2}(\sigma - 1) \\
    g_{2}(\sigma - 1) \\
    g_{2}(\sigma - 1) - g_{2}(2 - \sigma)
  \end{pmatrix}
  $
}
  \label{SO(even)MatrixEqn}
\end{equation}
for $\mathcal{G} = {\rm SO(even)}$ and
\begin{equation}
  \resizebox{.9\hsize}{!}{$
  \begin{pmatrix}
    \cos\left( \frac{\sigma - 1}{\sqrt{2}} \right) & \sin\left( \frac{\sigma - 1}{\sqrt{2}} \right) & 0 \\
    \cos\left( \frac{\sigma - 1}{\sqrt{2}} \right) & 0 & 0 \\
 \frac{-1}{\sqrt{2}}\sin\left( \frac{\sigma - 1}{\sqrt{2}} \right) + \cos\left( \frac{\sigma - 1}{\sqrt{2}} \right) & -\sqrt{2} + \frac{1}{\sqrt{2}}\cos\left( \frac{\sigma-1}{\sqrt{2}} \right) & -1
  \end{pmatrix}
  \begin{pmatrix}
    c_1 \\
    c_2 \\
    c_3
  \end{pmatrix}
  \ = \
\begin{pmatrix}
    g_{2}(\sigma - 1) \\
    g_{2}(\sigma - 1) \\
    g_{2}(\sigma - 1) - g_{2}(2 - \sigma).
  \end{pmatrix}
  $
}
  \label{SO(odd)AndSpMatrixEqn}
\end{equation}
for $\mathcal{G} = {\rm SO(odd)}$ or ${\rm Sp}$.  Here, $g_{2}$ is $g_{\mathcal{G}}$ restricted to $J_{2}$. \\

Expanding these matrices along the their third columns, we see that
\begin{equation}
  \left| A_{{\rm SO(even)}} \right| \ = \  \left| A_{\textrm{SO(odd)/Sp}} \right| \ = \  \cos\left( \frac{\sigma - 1}{\sqrt{2}} \right) \sin \left( \frac{\sigma - 1}{\sqrt{2}} \right),
  \label{The3x3Determinants}
\end{equation}
which are both nonzero for $1 < \sigma < 1.5$. Solving the matrix equations, we obtain
\begin{eqnarray}
  c_{1,{\rm SO(even)}} &\ = \ &  \cos \left(\frac{\sigma-1}{2}+\frac{1}{4} (-1-\pi )\right) \sec \left(\frac{\sigma-1}{\sqrt{2}}\right) \notag  \\
  c_{2,{\rm SO(even)}} &\ = \ & 0 \notag\\
  c_{3,{\rm SO(even)}} &\ = \ & \sin \left(\frac{1}{4} (2 \sigma+3 \pi -3)\right)+\frac{\sin \left(\frac{1}{4} (-2 \sigma+3 \pi +3)\right) \tan \left(\frac{\sigma-1}{\sqrt{2}}\right)}{\sqrt{2}}  \label{ActualCoefficientsSOEven}
\end{eqnarray}
and
\begin{eqnarray}
  c_{1,\mathcal{G}} &\ = \ & \cos \left(\frac{1-\sigma}{2}+\frac{1-\pi }{4}\right) \sec \left(\frac{\sigma-1}{\sqrt{2}}\right) \nonumber\\
  c_{2,\mathcal{G}} &\ = \ & 0 \notag \nonumber\\
  c_{3,\mathcal{G}} &\ = \ & \sin \left(\frac{1}{4} (-2 \sigma+3 \pi +3)\right)-\frac{\sin \left(\frac{1}{4} (2 \sigma+3 \pi -3)\right) \tan \left(\frac{\sigma-1}{\sqrt{2}}\right)}{\sqrt{2}}
  \label{ActualCoeffSOOdd/Sp}
\end{eqnarray}
for $\mathcal{G} = {\rm SO(odd)}$ or ${\rm Sp}$.

We currently have $(I + K_{\mathcal{G},\sigma})(\widetilde{g}_{\mathcal{G}}) = c$ for some constant $c$. Here, $\widetilde{g}_{\mathcal{G}}$ is the unscaled optimal function. As some of our $c_{i,\mathcal{G}}$ are nonzero and the operator $(I + K_{\mathcal{G}}, \sigma)$ is positive definite, this constant is nonzero and it can therefore be scaled to be one. We find the correct scaling factor by computing $((I + K_{\mathcal{G},\sigma})(\widetilde{g}_{\mathcal{G}})(0))^{-1}$, setting that equal to $\lambda_{\mathcal{G}}$ in \eqref{SOEvenExplicitOptimalFunction2} and \eqref{SOOdd/SpExplicitOptimalFunction2}. From these computations, we find
  \begin{eqnarray}
    \lambda_{\mathcal{G}, \sigma} & \ = \  &\widetilde{g}_{\mathcal{G}}(0) + \frac{1}{2}\int_{-1}^{1} \widetilde{g}_{\mathcal{G}}(y) dy \nonumber\\
    &\ = \ &   2 \sqrt{2} \sin \left(\frac{1}{4} (3-2 \sigma)\right)+(\sigma-1) \sin \left(\frac{1}{4} (-2 \sigma+\pi +3)\right) \nonumber\\
    & & \ \ \ \ +\ \frac{1}{2} \sin \left(\frac{1}{4} (2 \sigma+\pi -3)\right) \left(\sqrt{2} (s+1) \tan \left(\frac{s-1}{\sqrt{2}}\right)+2\right) \label{SO(even)Scaling}
  \end{eqnarray}
  for $\mathcal{G} =   {\rm SO(even)}$, and
  \begin{eqnarray}
    \lambda_{\mathcal{G},\sigma} &\ = \ & \widetilde{g}_{\mathcal{G}}(0) - \frac{1}{2} \int_{-1}^{1} g_{\mathcal{G}}(y) dy \nonumber\\
    &\ = \ &  -2 \sqrt{2} \sin \left(\frac{1}{4} (3-2 \sigma)\right)+(\sigma-1) \cos \left(\frac{1}{4} (2 \sigma+3 \pi -3)\right) \nonumber\\
    & & \ \ \ \ +\ \frac{1}{2} \sin \left(\frac{1}{4} (-2 \sigma+\pi +3)\right) \left(\sqrt{2} (\sigma-3) \tan \left(\frac{\sigma-1}{\sqrt{2}}\right)+2\right)
     \label{SpScaling}
  \end{eqnarray}
  for $\mathcal{G} = {\rm Sp}$, and
\begin{eqnarray}
  \lambda_{\mathcal{G}} &\ = \ & \lambda_{Sp} + \int_{-\sigma}^{\sigma} \widetilde{g}_{\mathcal{G}}(y) dy  \nonumber\\
  &\ = \ & \lambda_{Sp} + 4 (\sigma-1) \sin \left(\frac{1}{4} (2 \sigma+\pi -3)\right)+4 \sqrt{2} \sin \left(\frac{1}{4} (3-2 \sigma)\right) \nonumber\\
  & & \ \ \ \ -2 \sqrt{2} (\sigma-2) \sin \left(\frac{1}{4} (-2 \sigma+\pi +3)\right) \tan \left(\frac{\sigma-1}{\sqrt{2}}\right)
 \label{SOOddScaling}
\end{eqnarray}
for $\mathcal{G} = {\rm SO(odd)}$, completing the proof.
\end{proof}

\section{Conclusion and Future Work}\label{sec:conclusion}

In \cite{F} we include similar calculations for $3 < \sigma < 4$. We conjecture that these methods will provide solutions for all $\sigma \in \R^{+}$. In this pursuit, there are two important steps. The first is solving the system of delay differential equations. This gives a family of solutions. The second is taking the output of that system and finding the correct system of equations that give us the coefficients to pick our optimal $g$ out of that family. Preliminary calculations suggest the second step will be more difficult than the first; however, even
solving the first problem in general, or providing an algorithmic approach, is important progress as it reduces the problem to an optimization over a finite-dimensional space, as opposed to an infinite dimensional one.

While currently the best results (assuming no more than GRH) for showing agreement between number theory and random matrix theory for the 1-level density are only for ${\rm supp}(\hphi) \subset (-2, 2)$, in some cases larger support is attainable through additional assumptions. One example is the slight improvement for cuspidal newforms in \cite{ILS} under their Hypothesis S. Another are families of Dirichlet characters, where Fiorilli-Miller \cite{FiM} improve up to $(-4, 4)$ under some weak assumptions about the distribution of primes in residue classes (with stronger ones arbitrarily large support is attained). Thus there are already situations where we can gainfully employ these new optimal test functions in these expanded regimes.

Additionally in \cite{F} we hope to generalize these arguments to the $n$-level densities, and then either there or in future work examine the slight modifications needed in the optimal function if we have lower order terms.


\ \\

\end{document}